\documentclass{amsart}
\usepackage{amsthm}
\usepackage{amsmath}
\usepackage{amstext}
\usepackage{amsfonts}
\usepackage{latexsym}
\usepackage{amscd}
\usepackage{xypic}
\newcommand{\la}{\ensuremath{\rightarrow}}
\theoremstyle{plain}
\newtheorem{theorem}{Theorem}[section]
\newtheorem{lemma}[theorem]{Lemma}
\newtheorem{proposition}[theorem]{Proposition}
\newtheorem{corollary}[theorem]{Corollary}

\theoremstyle{definition}
\newtheorem{definition}[theorem]{Definition}
\newtheorem{definition-proposition}[theorem]{Definition-Proposition}

\newtheorem{example}[theorem]{Example}
\newtheorem{remark}[theorem]{Remark}

\newtheorem{question}[theorem]{Question}
\numberwithin{equation}{theorem}

\begin{document}

\title[Smoothings of normal crossing Fano schemes.]{Smoothings of Fano varieties with normal crossing singularities.}
\author{Nikolaos Tziolas}
\address{Department of Mathematics, University of Cyprus, P.O. Box 20537, Nicosia, 1678, Cyprus}
\email{tziolas@ucy.ac.cy}

\subjclass[2000]{Primary 14D15, 14D06, 14J45.}


\keywords{Algebraic geometry}

\begin{abstract}
This paper obtains criteria for a Fano variety $X$ defined over an algebraically closed field of characteristic zero with normal crossing singularities to be smoothable. In particular, we show that $X$ is smoothable by a flat deformation $\mathcal{X} \rightarrow \Delta$ with smooth total space $\mathcal{X}$ if and only if $T^1_X\cong \mathcal{O}_D$, where $D$ is the singular locus of $X$.
\end{abstract}

\maketitle

\section{Introduction}
This paper studies the deformation theory of a Fano variety defined over an algebraically closed field of characteristic zero with normal crossing singularities. In particular it is investigated when such a variety is smoothable. This means that there is a flat  projective morphism $ f \colon \mathcal{X} \la \Delta$, where $\Delta$ is the spectrum of a discrete valuation ring $(R,m_R)$, such that $\mathcal{X} \otimes_R (R/m_R) \cong X$ and $\mathcal{X} \otimes_R K(R)$ is smooth over the function field $K(R)$ of $R$. Moreover, it is investigated when such a smoothing exists with smooth total space $\mathcal{X}$. In this case $X$ is said to be totally smoothable.

Normal crossing singularities appear quite naturally in any degeneration problem. Let $f \colon \mathcal{X} \la C$ be a flat projective morphism from a variety $\mathcal{X}$ to a curve $C$. Then, according to Mumford's semistable reduction theorem~\cite{KKMS73},  after a finite base change and a birational modification the family can be brought to standard form $f^{\prime} \colon \mathcal{X}^{\prime} \la C^{\prime}$, where $\mathcal{X}^{\prime}$ is smooth and the special fibers are simple normal crossing varieties.

Smoothings of Fano varieties play a fundamental role in higher dimensional birational geometry as well. The outcome of the minimal model program starting with a smooth $n$-dimensional projective variety $X$, is a $\mathbb{Q}$-factorial terminal projective variety $Y$ such that either $K_Y$ is nef, or $Y$ has a Mori fiber space structure. This means that there is a projective morphism $f \colon Y \la Z$ such that $-K_Y$ $f$-ample, $Z$ is normal and $\dim Z \leq \dim X -1$. Suppose that the second case happens and $\dim Z=1$. Let $z \in Z$ and $Y_z= f^{-1}(z)$. Then $Y_z$ is a Fano variety of dimension $n-1$ and $Y$ is a smoothing of $Y_z$. The singularities of the special fibers are difficult to describe but normal crossing singularities naturally occur and are the simplest possible non-normal singularities.

Moreover, the study of smoothings $f \colon \mathcal{X} \la \Delta$ such that $\mathcal{X}$ is smooth, $-K_{\mathcal{X}}$ is $f$-ample and the special fiber is a simple normal crossing divisor, has a central role in the classification of smooth Fano varieties~\cite{Fu90}. In dimension two T. Fujita~\cite{Fu90} has described all the possible degenerations of smooth Del Pezzo surfaces to simple normal crossing Del Pezzo surfaces and Y. Kachi showed that all these actually occur~\cite{Kac07}. As far as I know this problem is completely open in higher dimensions.

It is therefore of interest to study which Fano varieties with normal crossing singularities are smoothable and in particular, which are totally smoothable.

The paper is organized as follows.

In section 3 we review basic properties of the theory of logarithmic structures and logarithmic deformations developed by K. Kato~\cite{Ka88} and F. Kato~\cite{Kat96} in the algebraic case and Y. Kawamata, Y. Namikawa~\cite{KawNa94} in the complex analytic case. These notions are essential in the investigation of when a Fano variety with normal crossing singularities is totally smoothable. The point is that sometimes singular varieties admit logarithmic structures in such a way that they become smooth in the log category. Moreover, deformations of varieties with smooth log structures behave like deformations of smooth varieties and therefore have very good deformation theory. Of course not all varieties admit smooth logarithmic structures. However, a variety $X$ with normal crossing singularities admits a semistable logarithmic structure and becomes log smooth if and only if $T^1_X=\mathcal{O}_D$, where $D$ is the singular locus of $X$ (Theorem~\ref{d-semistable}), which is exactly the case when $X$ is totally smoothable, and the reason why the theory of logarithmic deformations is so useful in the investigation of when a variety with normal crossing singularities is totally smoothable. 

In section 4 we study the obstruction spaces to deform a Fano variety $X$ with normal crossing singularities defined over an algebraically closed field of characteristic zero. It is well known that $H^2(T_X)$ and $H^1(T^1_X)$ are obstruction spaces to deformations of $X$. If $X$ has a semistable logarithmic structure then Theorem~\ref{log-obstructions} shows that $H^2(\mathcal{H}om_X(\Omega_X(\log),\mathcal{O}_X))$ is an obstruction space to logarithmic deformations. If $X$ is simple normal crossing, which means that $X$ has smooth irreducible components, then its obstruction theory is deeply clarified by the work of Friedman~\cite{Fr83}. However, in the general case when $X$ is not necessarily reducible, Friedman's theory does not directly apply. In Theorem~\ref{ob4} we show that if $X$ is a Fano variety with normal crossing singularities then $H^2(T_X)=0$. Moreover, if $X$ admits a semistable logarithmic structure, then $H^2(\mathcal{H}om_X(\Omega_X(\log),\mathcal{O}_X))=0$ and hence $X$ has unobstructed logarithmic deformations. However, usual deformations can be obstructed since the other obstruction space $H^1(T^1_X)$ may not vanish. This is the case in example~\ref{ex2}. However, $T^1_X$ is a line bundle on the singular locus $D$ of $X$ and in order for $X$ to be smoothable one has to impose some positivity conditions on $T^1_X$ that will force it to vanish. If $X$ has at worst double points, then in Theorem~\ref{db-ob} we show that $H^1(T^1_X)=0$ and hence $X$ unobstructed deformations in this case.

In sections 5 we apply the results of the previous sections to obtain criteria for the existence of a smoothing of a Fano variety $X$ with normal crossing singularities. We also study the problem of when $X$ is totally smoothable. Proposition~\ref{sm1} shows that if $X$ is totally smoothable, then $T^1_X \cong \mathcal{O}_D$, where $D$ is the singular locus of $X$. Therefore by Proposition~\ref{d-semistable}, $X$ has a logarithmic structure and the theory of logarithmic deformations applies in this case. The main result of section 5 is the following.

\begin{theorem}
Let $X$ be a Fano variety defined over an algebraically closed field of characteristic zero with normal crossing singularities. Assume that one of the following conditions hold:
\begin{enumerate}
\item $T^1_X$ is finitely generated by global sections and that $H^1(T^1_X)=0$.
\item $X$ has at worst double point normal crossing singularities and that $T^1_X$ is finitely generated by global sections.
\item $X$ is $d$-semistable, i.e., $T^1_X\cong \mathcal{O}_D$, where $D$ is the singular locus of $X$.
\end{enumerate}
Then $X$ is smoothable. Moreover, $X$ is smoothable by a flat deformation $f \colon \mathcal{X} \rightarrow \Delta$ such that $\mathcal{X}$ is smooth, if and only if $X$ is $d$-semistable.
\end{theorem}

I do not know if the condition that $T^1_X$ is finitely generated by its global sections is a necessary condition too for $X$ to be smoothable. $X$ is certainly not smoothable of $H^0(T^1_X)=0$~\cite{Tz10}. In all the cases of the previous theorem the condition finitely generated by global sections implies that $\mathrm{Def}(X)$ is smooth. If it is true that $\mathrm{Def}(X)$ is smooth for any $X$, then $X$ smoothable implies that $T^1_X$ is finitely generated by its global sections too.

In section 6 we give an example of a smoothable and one of a non-smoothable Fano threefold.

Finally, the requirement that we work over an algebraically closed field is more technical than essential. In the general case I believe that the definition of normal crossing singularities must be modified to allow singularities like $x_0^2+x_1^2=0$ in $\mathbb{R}^2$. This would make the arguments more complicated without adding anything of essence to the proofs. However, the characteristic zero assumption is essential since we make repeated use of the Akizuki-Kodaira-Nakano vanishing theorem.

\section{Terminology-Notation.}

All schemes in this paper are defined over an algebraically closed field $k$.

A reduced scheme $X$ of finite type over $k$ is said to have normal crossing (n.c.) singularities at a point $P\in X$ if $\hat{\mathcal{O}}_{X,P} \cong k(P)[[x_0,\dots,x_n]]/(x_0\cdots x_r)$, for some $r=r(P)$, where $k(P)$ is the residue field of $\mathcal{O}_{X,P}$ and $\hat{\mathcal{O}}_{X,P}$ is the completion of $\mathcal{O}_{X,P}$ at its maximal ideal. If $r=2$ then we say that $P\in X$ is a double point normal crossing singularity. $X$ is called a normal crossing variety if it has normal crossing singularities at every point. In addition, if $X$ has smooth irreducible components then it is called a simple normal crossing variety (s.n.c.).

A reduced projective scheme $X$ with normal crossing singularities is called a Fano variety if and only if $\omega_X^{-1}$ is an ample invertible sheaf on $X$.

For any scheme $X$ we denote by $T^1_X$ the sheaf of infinitesimal first order deformations of $X$~\cite{Sch68}. If $X$ is reduced then $T^1_X=\mathcal{E}xt^1_X(\Omega_X,\mathcal{O}_X)$. If $X$ has n.c. singularities then a straightforward local calculation shows that $T^1_X$ is a line bundle on the singular locus $D$ of $X$. Moreover, if $X=\cup_{i=1}^NX_i$ is a simple normal crossings variety, then~\cite{Fr83}
\[
T^1_X=\mathcal{H}om_D((I_{X_1}/I_{X_1}I_D) \otimes \cdots  \otimes (I_{X_k}/I_{X_k}I_D),\mathcal{O}_D) 
\]

A variety $X$ with normal crossing singularities is called $d$-semistable if and only if $T^1_X \cong \mathcal{O}_D$, where $D$ is the singular locus of $X$. 

We say that $X$ is smoothable if there is a flat morphism of finite type $f \colon \mathcal{X} \rightarrow \Delta$, $\Delta=\mathrm{Spec}(R)$, where $R$ is a discrete valuation ring, such that the cental fiber $\mathcal{X}_0$ is isomorphic to $X$ and the general fiber $\mathcal{X}_g$ is smooth over the function field $K(R)$ of $R$. 

Finally, we will repeatedly make use of the Akizuki-Kodaira-Nakano vanishing theorem and its logarithmic version, which we state next.
\begin{theorem}[Akizuki-Kodaira-Nakano~\cite{AN54},~\cite{EV92}]\label{KN}
Let $X$ be a smooth variety defined over an algebraically closed field of characteristic zero and $\mathcal{L}$ an ample invertible sheaf on $X$. Then
\item
\[
H^b(X,\Omega_X^a\otimes \mathcal{L}^{-1})=0
\]
for all $a,b$ such that $a+b < \dim X$.
\item Moreover, if $D$ is a reduced simple normal crossings divisor of $X$, then
\[
H^b(X,\Omega_X^a(\mathrm{log}(D))\otimes \mathcal{L}^{-1})=0
\]
for all $a,b$ such that $a+b < \dim X$.
\end{theorem} 
\section{Logarithmic Structures.}

In order to study the deformation theory of certain Fano varieties we will use the theory of logarithmic structures and deformations which  was developed by K. Kato~\cite{Ka88} and F. Kato~\cite{Kat96} in the algebraic case and Y. Kawamata, Y. Namikawa in the complex analytic case~\cite{KawNa94}. For the convenience of the reader we make a short review of basic properties and results that will be used in this paper and refer the reader to the aforementioned papers for more details. 

\begin{definition}
Let $X$ be a scheme. A pre-logarithmic structure on $X$ is a sheaf of monoids $\mathcal{M}$ on the \'etale site $X_{et}$ together with a sheaf of monoids homomorphism $ \alpha \colon \mathcal{M} \rightarrow \mathcal{O}_X$, with respect to the multiplication of $\mathcal{O}_X$. A pre-logarithmic structure is called a logarithmic structure if $\alpha^{-1}(\mathcal{O}_X^{\ast}) \cong \mathcal{O}_X^{\ast}$. 

For simplicity, from now on pre-logarithmic structures will be called pre-log structures and logarithmic structures will be called log-structures.
\end{definition}

A morphism $(X,\mathcal{M}) \rightarrow (Y,\mathcal{N})$ of schemes with log structures is a pair $(f,g)$, where $f \colon X \rightarrow Y$ is a scheme morphism, $g \colon f^{-1}(\mathcal{N}) \rightarrow \mathcal{M}$ a sheaf of monoids map such that the following diagram commutes
\[
\xymatrix{
 f^{-1}(\mathcal{N})\ar[d]\ar[r]^g &  \mathcal{M}\ar[d] \\
   f^{-1}(\mathcal{O}_Y)\ar[r] & \mathcal{O}_X
}
\]
To any pre-log structure $(\mathcal{M},\alpha)$ on a scheme $X$ there is a naturally defined log structure $(M^a,\alpha)$ which is universal for homomorphisms of pre-log structures from $\mathcal{M}$ to log-structures of $X$. Moreover, given a scheme morphism $f \colon X \rightarrow Y$ and log structures $\mathcal{M}, \mathcal{N}$ on $X$ and $Y$ respectively, the preimage and direct image log structures $f^{\ast}\mathcal{N}$ and $f_{\ast}\mathcal{M}$ are also naturally defined.

Let $\mathcal{M}$ be a log structure on a scheme $X$. The log structure is called integral if $\mathcal{M}$ is a sheaf of integral monoids and fine if \'etale locally on $X$ there is a finitely generated monoid $P$ and a sheaf of monoids map $P_X \rightarrow \mathcal{O}_X$, where $P_X$ is the constant sheaf associated to $P$, such that the log structure associated to $P_X$ is $\mathcal{M}$.  

A morphism $(X,\mathcal{M}) \rightarrow (Y,\mathcal{N})$ of schemes with fine log structures is called smooth if it satisfies a logarithmic version of the infinitesimal criterion of smoothness. It is a natural extension of the usual notion of smoothness in the category of schemes with fine log structures. An interesting part of the theory is that morphisms that are not smooth in the category of schemes, become smooth in the log-category, with suitably chosen log-structures. Example~\ref{ss-example} exhibits such a case.

\subsection{Log differentials, log derivations and log deformations.}

There is a natural extension of differentials, derivations and deformations in the log category. 
\begin{definition}
\begin{enumerate}
\item Let $f \colon (X,\mathcal{M}) \rightarrow (Y,\mathcal{N})$ be a morphism of schemes with fine log-structures and let $\mathcal{E}$ be an $\mathcal{O}_X$-module. A log-derivation from $(X,\mathcal{M})$ to $\mathcal{E}$ over $(Y,\mathcal{N})$ is a pair $(D,D\log)$, where $D \in Der_Y(X,\mathcal{E})$ is a usual derivation and $D\log \colon \mathcal{M} \rightarrow \mathcal{E}$ is a map such that
\begin{enumerate}
\item $D\log(ab)=D\log(a)+D\log(b)$, for $a,b \in \mathcal{M}$,
\item $D(\alpha(a))=\alpha(a)D\log(a)$, for $a \in \mathcal{M}$,
\item $D\log(\phi(c))=0$, for all $c \in f^{-1}\mathcal{N}$, where $\phi \colon f^{-1} \rightarrow \mathcal{M}$ is the sheaf of monoids map associated to the morphism $f$.
\end{enumerate}  
\item The sheaf of log-differentials of $(X,\mathcal{M})$ over $(Y,\mathcal{N})$ is the $\mathcal{O}_X$-module $\Omega_{X/Y}(\log(\mathcal{M}/\mathcal{N})$ defined by
\[
\Omega_{X/Y}(\log(\mathcal{M}/\mathcal{N})) = \frac{\Omega_{X/Y}\oplus (\mathcal{O}_X\otimes_{\mathbb{Z}}\mathcal{M}^{gp})}{K}
\]
where $K$ is the $\mathcal{O}_X$-submodule generated by $(d\alpha(a),0)-(0,\alpha(a)\otimes a)$ and $(0,1\otimes \phi(b))$, for all $a \in \mathcal{M}$, $b \in \mathcal{N}$.
\end{enumerate}
\end{definition} 

Next we define the notion of a log-deformation. Let $\Lambda$ be a complete Noetherian local ring with maximal ideal $m_{\Lambda}$ and residue field $k$. Let $Q$ a fine saturated monoid having no invertible element other than $1$. $Q$ defines a log-structure $(\mathrm{Spec}k,Q)$. Let $\Lambda[[Q]]$ be the completion of the monoid ring $\Lambda[Q]$ along the maximal ideal $m_{\Lambda}+\Lambda[Q-\{1\}]$ (if $Q=\mathbb{N}$, $\Lambda[[Q]]=\Lambda[[t]]$). The map $Q \rightarrow \Lambda[[Q]]$ defines a log-structure on $\Lambda[[Q]]$ and on any Artin local $\Lambda$-algebra $A$, via the $\Lambda[[Q]]$-algebra map $\Lambda[[Q]] \rightarrow A$. Let $Art_{\Lambda[[Q]]}(k)$ be the category of local Artin $ \Lambda[[Q]]$-algebras. 

\begin{definition}
Let $f \colon (X,\mathcal{M})\rightarrow (\mathrm{Spec}k,Q)$ be a log-smooth morphism and $A \in Art_{\Lambda[[Q]]}(k)$. A log-smooth deformation of $f$ over $A$ is a a cartesian diagram
\[
\xymatrix{ 
(X,\mathcal{M}) \ar[r]\ar[d]^f & (X_A,\mathcal{M}_A) \ar[d]^{f_A} \\
(\mathrm{Spec}k,Q) \ar[r] & (\mathrm{Spec}(A),Q)
}
\]
where $f_A \colon (X_A,\mathcal{M}_A) \rightarrow (\mathrm{Spec}(A),Q)$ is log-smooth. In particular, if $Q=\mathbb{N}$, then the underlying scheme morphisms are flat and hence a usual deformation.
\end{definition}
Having defined log-deformations, the log-deformation functor \[
LD(X,\mathcal{M}) \colon Art_{\Lambda[[Q]]}(k) \rightarrow (Sets)
\]
is naturally defined. 
\begin{theorem}[Theorem 8.7~\cite{Kat96}]\label{hull}
If $f \colon (X\mathcal{M})\rightarrow (\mathrm{Spec}k,Q)$ is integral and $X$ proper, then $LD(X,\mathcal{M})$ has a hull.
\end{theorem}

The log-deformation theory of log-smooth maps is very similar to the deformation theory of smooth varieties. The next theorem describes the obstructions to lift log-smooth deformations. 
\begin{theorem}[K.Kato Theorem 3.14~\cite{Ka88}]\label{log-obstructions}
Let $f_A \colon (X_A,\mathcal{M}_A) \rightarrow (\mathrm{Spec}(A),Q)$ be a log-deformation of the log-smooth map $f \colon (X,\mathcal{M})\rightarrow (\mathrm{Spec}k,Q)$. Let \[
0 \rightarrow I \rightarrow B \rightarrow A \rightarrow 0
\]
be a square zero extension in $Art_{\Lambda[[Q]]}(k)$. Then the obstructions for lifting $f_A$ to $B$ are in 
\[
H^2(X_A,\mathcal{H}om_{{X_A}}(\Omega_{{X_A}/A}(\log(\mathcal{M}_A /Q))))\otimes_A I
\]
\end{theorem}

\subsection{Logarithmic structures on varieties with normal crossing singularities.}

Next we present some logarithmic structures on schemes with normal crossing singularities which will be needed for the study of smoothability of Fano varieties with normal crossing singularities.

\begin{example}\label{emb-example}
Let $D \subset X$ be a reduced divisor with normal crossings in a smooth scheme $X$. Let $M \subset \mathcal{O}_{X}$ be the subsheaf of $\mathcal{O}_{X}$ of regular functions that are invertible outside of $D$. $M$ is a log-structure on $X$. Moreover, if $i \colon D \rightarrow X$ is the closed immersion, then $i^{\ast}M$ is a log-structure on $D$. This log-structure is called of embedding type and it is fine because \'etale locally around $D$, $D \subset X$ is just $(x_1\cdots x_r=0) \subset \mathbb{A}^d_k$ and the log-structure is induced by the monoid map
\[
\alpha \colon \mathbb{N}^d \rightarrow \frac{k[x_1,\ldots,x_d]}{(x_1\cdots x_r)}
\]
given by $\alpha(e_i)= x_i$, if $i \leq r$, and $1$ for $r< i \leq d$.

\end{example}

\begin{example}\label{ss-example}
Let \[
X=\mathrm{Spec} \frac{k[x_1,\ldots,x_d]}{(x_1\cdots x_r)}
\]
be a simple normal crossing variety and $\mathbb{N}^d \rightarrow \mathcal{O}_X$ the log structure defined in the previous example. The map $\beta \colon \mathbb{N} \rightarrow k$ such that $\beta(0)=1$ and $\beta (n)=0$, for $n \not=0$ defines a log-structure on $\mathrm{Spec}k$. Let $\delta \colon \mathbb{N} \rightarrow \mathbb{N}^d$ be the diagonal map. Then the corresponding map of log-schemes $(X,\mathbb{N}^d) \rightarrow (\mathrm{Spec}k, \mathbb{N})$ is log smooth and is called a logarithmic semi-stable map.

\end{example}

\begin{definition}
Let $\alpha \colon \mathcal{M} \rightarrow \mathcal{O}_X$ be a log-structure on a scheme $X$ with normal crossing singularities.
\begin{enumerate} 
\item The log-structure is called of embedding type if locally in the \'etale topology it is equivalent to the log-structure of embedding type defined in example~\ref{emb-example}.
\item The log-structure is called of semistable type if there is a map of log schemes $f \colon (X,\mathcal{M}) \rightarrow (\mathrm{Spec}k, \mathbb{N})$ which is locally in the \'etale topology equivalent to the logarithmic semistable map defined in example~\ref{ss-example}. 
\end{enumerate}
In the case of a scheme with semistable log-structure as above, we denote for simplicity by $\Omega_X(\log)$ the sheaf of logarithmic differentials of $(X,\mathcal{M})$ over $(\mathrm{Spec}k,\mathbb{N})$. 
$\Omega_{X}(\log)$ is a free $\mathcal{O}_X$-module locally generated by the logarithmic differentials 
\[
\frac{dx_1}{x_1},\ldots, \frac{dx_r}{x_r},dx_{r+1},\ldots,dx_n
\]
with the relation
\[
\frac{dx_1}{x_1}+\cdots + \frac{dx_r}{x_r}=0.
\]

The existence of logarithmic structures of embeded type or semistable type is very closely related to the deformation theory of $X$.

\begin{theorem}[F. Kato, Theorem 11.7~\cite{Kat96}]\label{d-semistable}
Let $X$ be a scheme with normal crossing singularities and $D$ its singular locus.
\begin{enumerate}
\item A log-structure of embedding type exists on $X$ if and only if there exists a line bundle $L$ on $X$ such that $L\otimes_{\mathcal{O}_X}\mathcal{O}_D \cong T^1_X$.
\item A log-structure of semistable type exists on $X$ if and only if $T^1_X \cong \mathcal{O}_D$, i.e., if $X$ is $d$-semistable.
\end{enumerate}
\end{theorem}

\end{definition}
\section{Obstructions}
Let $X$ be a variety with normal crossing singularities. It is well known that $H^2(T_X)$ and $H^1(T^1_X)$ are obstruction spaces to deformations of $X$. The purpose of this section is to describe these spaces and moreover to describe the obstruction space $H^2(\mathcal{H}om(\Omega_X(log),\mathcal{O}_X))$ to logarithmic deformations of a n.c. variety $X$ with a semistable logarithmic structure. We begin with some preliminary results.

\begin{proposition}[~\cite{Fr83}]\label{ob1}
Let $X$ be a scheme with only normal crossing singularities. Let $\tau_X \subset \Omega_X$ be the torsion subsheaf of $\Omega_X$. Then for all $i \geq 0$,
\begin{gather*}
\Omega_X/\tau_X \cong \Omega_X^{\ast\ast} \\
\mathrm{Ext}_X^i(\Omega_X/\tau_X,\mathcal{O}_X)\cong H^i(T_X) \\
\mathrm{Ext}_X^i(\tau_X,\mathcal{O}_X)=H^{i-1}(T^1_X)
\end{gather*}
\end{proposition}

\begin{corollary}\label{ob2}
Let $X$ be a projective scheme with only normal crossing singularities. Then
\[
H^2(T_X)=H^{n-2}((\Omega_X/\tau_X) \otimes \omega_X)
\]
where $\dim X =n$.
\end{corollary}

\begin{proof}
By Proposition~\ref{ob1}, $H^2(T_X)=\mathrm{Ext}_X^2(\Omega_X/\tau_X,\mathcal{O}_X)=H^{n-2}((\Omega_X/\tau_X) \otimes \omega_X)$, by Serre duality.
\end{proof}

\begin{definition}\label{sing-k}
Let $X$ be a scheme with normal crossing singularities defined over a field $k$. Then we denote by $X_{[k]} \subset X$, $k \geq 0$, the subschemes of $X$
defined inductively by $X_{[0]}=X$ and $X_{[k]}$ the singular locus of $X_{[k-1]}$ with reduced structure. We also denote by $\pi_i \colon \tilde{X}_{[i]} \rightarrow X_{[i]}$ the normalization of $X_{[i]}$ $i \geq 0$.
\end{definition}

\newpage

\begin{theorem}\label{ob3}
Let $X$ be a scheme with normal crossing singularities defined over a field $k$. Then,
\begin{enumerate}
\item There are exact sequences
\begin{gather}\label{ob-seq}
0 \la \Omega_X/\tau_X \la (\pi_0)_{\ast}\Omega_{\tilde{X}} \stackrel{\delta_1}{\la}  (\pi_1)_{\ast}(\Omega_{\tilde{X}_{[1]}}\otimes L_1)
\stackrel{\delta_2}{\la} \cdots \stackrel{\delta_N}{\la}(\pi_N)_{\ast}(\Omega_{\tilde{X}_{[N]}}\otimes L_N) \la 0 \\
0 \la \mathcal{O}_X \rightarrow (\pi_0)_{\ast}\mathcal{O}_{\tilde{X}} \rightarrow (\pi_1)_{\ast}Q_1 \rightarrow \cdots \rightarrow (\pi_N)_{\ast}Q_N \rightarrow 0 
\end{gather}
\item Suppose that $X$ has a semistable logarithmic structure. Then there is an exact sequence
\begin{equation}\label{log-seq1}
0 \rightarrow \Omega_X /\tau_X \rightarrow \Omega_{X}(log) \stackrel{\lambda_1}{\rightarrow} (\pi_1)_{\ast}(\mathcal{O}_{\tilde{X}_{[1]}} \otimes M_1) \stackrel{\lambda_2}{\rightarrow} \cdots \stackrel{\lambda_m}{\rightarrow} (\pi_m)_{\ast}(\mathcal{O}_{\tilde{X}_{[m]}} \otimes M_m) \rightarrow 0
\end{equation}
\end{enumerate}
where $m,N \leq \dim X$ and $Q_i, L_i$, $M_i$ are 2-torsion invertible sheaves on $\tilde{X}_{[i]}$, i.e., $L_i^{\otimes 2} \cong \mathcal{O}_{\tilde{X}_{[i]}}$, $M_i^{\otimes 2} \cong \mathcal{O}_{\tilde{X}_{[i]}}$ and $Q_i^{\otimes 2} \cong \mathcal{O}_{\tilde{X}_{[i]}}$, for all $i$. 
\end{theorem}

\begin{remark}
In the case of simple normal crossing complex analytic spaces, Theorem~\ref{ob3} was proved by R. Friedman~\cite{Fr83}.
\end{remark}

The following result is needed for the proof of the theorem.

\begin{lemma}\label{normalization}
Let $f \colon Y \la X$ be an \'etale morphism of schemes. Let $ \pi \colon \tilde{X} \la X$ be the normalization of $X$. Then  $p_Y \colon \tilde{X} \times_X Y \la Y$ is the normalization of $Y$,
where $ \tilde{X} \times_X Y \la Y$ is the fiber product of $\tilde{X}$ and $Y$ over $X$ and $p_Y$ the projection to $Y$.
\end{lemma}
\begin{proof}
From the fiber square diagram
\[
\xymatrix{
\tilde{X} \times_X Y \ar[r]^{p_Y}\ar[d]_{p_{\tilde{X}}}  & Y \ar[d]^{f} \\
\tilde{X} \ar[r]^{\pi}& X
}
\]
it follows that $p_{\tilde{X}}$ is \'etale and $p_Y$ finite. Hence $\tilde{X} \times_X Y $ is normal. Moreover $p_Y$ is generically isomorphism. Therefore there is a factorization
$g \colon \tilde{X} \times_X Y \la \tilde{Y}$ of $p_Y$ through the normalization $\tilde{Y}$ of $Y$. But then, since both $\tilde{X} \times_X Y $, $\tilde{Y}$ are normal, $g$ finite and generically isomorphism, $g$ is in fact an isomorphism.
\end{proof}

\begin{proof}[Proof of Theorem~\ref{ob3}]
We will only prove the existence of the exact sequence~\ref{ob3}.1  in detail. The proof for the others is similar. We will only sketch it and leave the details to the reader.

The proof of the first part is in two steps. First we show the existence of the exact sequence~(\ref{ob-seq}) for an affine simple normal crossing scheme and then we prove the general case.
The proof of this part is similar to the one exhibited by Friedman~\cite{Fr83} in the case of a simple normal crossing complex analytic space. For
the sake of completeness, and since the explicit local construction of the sequence is needed for the general case, we present a short proof here
following the lines of Friedman's proof.

\textbf{Step 1.} Suppose \[
                          X=\mathrm{Spec}\frac{k[x_1, \ldots ,x_n]}{(x_1\cdots x_r)}
                         \]
Then $X=\cup_{i=1}^r X_i$, where $X_i \subset X$ is the component given by $x_i =0$, $1 \leq i \leq r$. Then for $i \geq 1$,
\[
X_{[i]} = \bigcup_{k_0 < \cdots < k_i} (X_{k_0} \cap \cdots \cap X_{k_i})
\]
Moreover, $\tilde{X}=\coprod_{i=1}^rX_i$ and
\[
\tilde{X}_{[i]}=\coprod_{k_0 < \cdots < k_i}(X_{k_0} \cap \cdots \cap X_{k_i})
\]
The maps $\pi_i \colon \tilde{X}_{[i]} \la X_{[i]}$, $i \geq 0$, are the natural ones. Now by definition,
$\tau_X \subset \Omega_X$ is the sheaf of sections of $\Omega_X$ supported on the singular locus of $X$. Hence it is the kernel of
the natural map $\delta \colon \Omega_X \rightarrow \pi_{\ast}\Omega_{\tilde{X}}$. Now define the sequence of maps
\begin{equation}\label{local-seq}
0 \la \tau_X \la \Omega_X \stackrel{\delta}{\la} (\pi_0)_{\ast}\Omega_{\tilde{X}} \stackrel{\delta_1}{\la}  (\pi_1)_{\ast}\Omega_{\tilde{X}_{[1]}}
\stackrel{\delta_2}{\la} \cdots \stackrel{\delta_i}{\la}(\pi_i)_{\ast}\Omega_{\tilde{X}_{[i]}}
\stackrel{\delta_{i+1}}{\la}(\pi_{i+1})_{\ast}\Omega_{\tilde{X}_{[i+1]}} \la \cdots
\end{equation}
where $\delta_i$ are the \v{C}ech coboundary maps. This is clearly a complex and we proceed to show that it is in fact exact. We use induction
on the number $r$ of components of $X$. For $r=1$ there is nothing to prove. Suppose now that the sequence~(\ref{local-seq}) is exact for all
simple normal crossing affine schemes with at most $r-1$ components.

Let $X^{\prime}=\cup_{i=1}^{r-1}X_i$ and $Y=X^{\prime} \cap X_r$. Then,
$X=X^{\prime} \cup X_r$, $\tilde{X}_{[k]}={\tilde{X}}^{\prime}_{[k]}\coprod \tilde{Y}_{[k-1]}$, for all $k \geq 0$, where we also set $Y_{[-1]}=X_r$.
By the induction hypothesis, the corresponding sequences~(\ref{local-seq}) for $X^{\prime}$ and $Y$, are exact.

From the previous discussion, it follows that $\mathrm{Ker}(\delta) = \tau_X$.
Next we show exactness at the next step, i.e., that $\mathrm{Ker}(\delta_1)=\mathrm{Im}(\delta)$. Now since
\begin{gather*}
(\pi_0)_{\ast}\Omega_{\tilde{X}}= (\pi_0)_{\ast}\Omega_{{\tilde{X}}^{\prime}}\oplus \Omega_{X_r}\\
(\pi_1)_{\ast}\Omega_{\tilde{X}_{[1]}}=(\pi_1)_{\ast}\Omega_{{\tilde{X}}_{[1]}^{\prime}}\oplus \Omega_{\tilde{Y}}
\end{gather*}
any element of $(\pi_0)_{\ast}\Omega_{\tilde{X}}$ is of the form $(\alpha,\beta)$, where
$\alpha \in (\pi_0)_{\ast}\Omega_{{\tilde{X}}^{\prime}}$ and $\beta \in \Omega_{X_r}$. Suppose that such an element is also in the kernel of $\delta_1$.
It is now clear from the induction hypothesis that $(\alpha,0)$ is in the image of $\delta$. Therefore, in order to show exactness at the level of
$\delta_1$, it suffices to show that if an element of the form $(0,\beta)$ is in the kernel of $\delta_1$, it is also in the image of $\delta$.
Suppose that $\beta=\sum_{k \not= r} \alpha_r(f_k) dx_k$ is such an element, where $f_k\in \mathcal{O}_{X}$ and $\alpha_r \colon \mathcal{O}_X \rightarrow
\mathcal{O}_{X_r}$ the natural map. Therefore, since
$(\pi_0)_{\ast}\Omega_{\tilde{Y}}= \oplus_{i=1}^{r-1}\Omega_{X_i \cap X_r}$,
it follows that the restriction of $\beta$ on $X_i \cap X_r$ is zero, for all $i \leq r-1$. Hence $f_k \in (x_1\cdots \hat{x}_k\cdots x_{r-1},x_r)$,
for $1 \leq k \leq r$, and $f_k \in (x_1\cdots x_{r-1},x_r)$, for $k > r$. Therefore, $\delta(\sum_{k \not= r} f_k dx_k)=(0,\beta)$ and hence
$(0, \beta)$ is in the image of $\delta$.

There is an exact sequence
\begin{equation}\label{complex}
0 \la (\pi_{k-1})_{\ast}\Omega_{\tilde{Y}_{[k-1]}} \la (\pi_{k})_{\ast}\Omega_{\tilde{X}_{[k]}} \la (\pi_{k})_{\ast}\Omega_{{\tilde{X}}^{\prime}_{[k]}} \la 0
\end{equation}
Now define the complexes $(A^{\ast}, \delta_A^{\ast})$, $(B^{\ast}, \delta_B^{\ast})$ and $(C^{\ast}, \delta_C^{\ast})$, such that
$A^k=(\pi_{k-1})_{\ast}\Omega_{\tilde{Y}_{[k-1]}}$, $B^k=(\pi_{k})_{\ast}\Omega_{\tilde{X}_{[k]}}$, and $C^k=(\pi_{k})_{\ast}\Omega_{{\tilde{X}}^{\prime}_{[k]}}$,
$k \geq 0$, and the coboundary maps are the \v{C}ech maps. Then~(\ref{complex}) induces an exact sequence of complexes
\[
0 \la A^{\ast} \la B^{\ast} \la C^{\ast} \la 0
\]
Passing to cohomology we get an exact sequence
\[
\cdots \la H^k(A^{\ast}) \la H^k(B^{\ast}) \la H^k(C^{\ast}) \la \cdots
\]
Now by induction $H^k(C^{\ast})=0$, for all $k \geq 1$ and $H^k(A^{\ast})=0$, for all $k \geq 2$. Hence $H^k(B^{\ast})=0$, for all $k \geq 2$. It remains
to check for $k=1$. Then there is an exact sequence
\[
H^0(C^{\ast}) \stackrel{\sigma}{\la} H^1(A^{\ast}) \la H^1(B^{\ast}) \la 0
\]
Moreover, $H^0(C^{\ast})=\Omega_{X^{\prime}}/\tau_{X^{\prime}}$, $H^1(A^{\ast})=\Omega_Y/\tau_Y$ and $\sigma$ is the natural map
\[
\Omega_{X^{\prime}}/\tau_{X^{\prime}} \la \Omega_Y/\tau_Y
\]
and hence it is surjective. Therefore $H^1(B^{\ast})=0$ and the complex~(\ref{local-seq}) is exact as claimed.

\textbf{Step 2.} The general case. So, let $X$ be a scheme with normal crossing singularities. 

\textit{Claim:} For any $x \in X$ there are pointed \'etale maps
\begin{equation}\label{et-neigh}
\xymatrix{
    &  (U,u) \ar[dl]_{f} \ar[dr]^{g} & \\
(X,x)   & & (W,w)
}
\end{equation}
where $u \in U$, $w\in W$, $f(u)=x$, $g(u)=w$ and such that:
\begin{enumerate}
\item $f$ and $g$ induce isomorphisms of residue fields $k(u)\cong k(x) \cong k(w)$.
\item  
\[
W=\mathrm{Spec}\frac{k[x_1,\ldots,x_n]}{(x_1\cdots x_{r(x)})}
\]
and $w \in W$ corresponds to the maximal ideal $(x_1,\ldots,x_n)$.
\item All irreducible components of $U$ pass through $u$. 
\item $U$ is a simple normal crossing scheme and it has exactly $r=r(x)$ irreducible components, exactly as many as $W$. Let $U = \cup_{k=1}^{r}U_{k}$ be the decomposition of $U$ into its irreducible components. Then there is an exact sequence
\begin{equation}\label{local-seq2}
0 \la \tau_{U} \la \Omega_{U} \stackrel{\delta}{\la} (\pi_{0})_{\ast}\Omega_{\tilde{U}} \stackrel{\delta_{1}}{\la}  (\pi_{1})_{\ast}\Omega_{\tilde{U}_{[1]}}
\stackrel{\delta_{2}}{\la} \cdots \stackrel{\delta_{k}}{\la}(\pi_{k})_{\ast}\Omega_{\tilde{U}_{[k]}}
\stackrel{\delta_{k+1}}{\la}\cdots
\end{equation}
where as before,
\[
U_{[k]}= \bigcup_{s_0 < \cdots < s_k} (U_{s_0} \cap \cdots \cap U_{ s_k}),
\]
$\tilde{U}_{[k]}$ is the normalization of $U_{[k]}$ and the boundary maps are the \v{C}ech maps.
\end{enumerate}
A diagram of maps as~\ref{et-neigh} that satisfies (1), (2) and (3) is called an \'etale neighborhood of $x \in X$. Note that the map $g$ induces an ordering on the irreducible components $U_i=g^{-1}W_i$ of $U$, $W_i$ being the irreducible component of $W$ given by $x_i=0$, $i=1,\ldots, r(x)$.
 
We proceed to show the claim. Let $x \in X$ a point. Then by assumption, 
\[
\hat{\mathcal{O}}_{X,x} \cong \frac{k[[x_1, \ldots, x_n]]}{(x_1\cdots x_{r(x)})}
\]
Let $W=\mathrm{Spec}(k[x_1,\ldots,x_n]/(x_1\cdots x_{r(x)}))$ and let $w$ be the closed point corresponding to the maximal ideal $(x_1, \ldots, x_n)$. Then by~\cite{Art69}, since $\hat{\mathcal{O}}_{X,x} \cong\hat{\mathcal{O}}_{W,w}$, there is a common \'etale neighborhood of $x \in X$ and $w \in W$, i.e., there are pointed \'etale maps $f \colon (U,u) \rightarrow (X,x)$ and $g \colon (U,u) \rightarrow (W,w)$ as in~\ref{et-neigh} that satisfy the properties (1) and (2) of the claim.

Shrinking $U$ we may assume that all its irreducible components pass through $u \in U$. Let $W=\cup_{i=1}^{r}W_{i}$, $r=r(x)$, be the decomposition of $W$ into irreducible components, where $W_{i}$ is given by $x_i=0$, $i=1, \ldots, r$. Then $U=\cup_{i=1}^{r}g^{-1}(W_{i})$. Since $g$ is \'etale and all irreducible components of $U$ pass through the same point, it follows that $U_i=g^{-1}(W_i)$ is smooth and irreducible and hence $U=\cup_{i=1}^r U_i$, in particular $U_i$ is simple normal crossings and has the same number of irreducible components as $W$.

From step 1. there is an exact sequence
\begin{equation}\label{complex1}
0 \la \tau_{W} \la \Omega_{W} \stackrel{\delta}{\la} (\pi_{w,0})_{\ast}\Omega_{\tilde{W}} \stackrel{\delta_{w,1}}{\la}  (\pi_{w,1})_{\ast}\Omega_{\tilde{W}_{[1]}}
\stackrel{\delta_{w,2}}{\la} \cdots 
\end{equation}
Since $g$ is \'etale, by Lemma~\ref{normalization} there is a fiber square diagram
\[
\xymatrix{
\tilde{U}_{[k]} \ar[r]^{\pi_{u,k}}\ar[d]_{\tilde{g}} & U_{[k]} \ar[d]^{g} \\
\tilde{W}_{[k]} \ar[r]^{\pi_{w,k}} & W_{[k]} \\
}
\]
By flat base change it follows that $g^{\ast}(\pi_{w,k})_{\ast}=(\pi_{u,k})_{\ast}(\tilde{g}_{p})^{\ast}$. Moreover, since both $g$ and $\tilde{g}$ are \'etale,
\begin{gather*}
g^{\ast}\Omega_{W_{[k]}}= \Omega_{U_{[k]}} \\
(\tilde{g})^{\ast}\Omega_{\tilde{W}_{[k]}}= \Omega_{\tilde{U}_{[k]}}
\end{gather*}
Therefore~(\ref{complex1}) pulls back via $g$ to an exact sequence in $U$,
\begin{equation}\label{complex2}
0 \la \tau_{U} \la \Omega_{U} \stackrel{\delta}{\la} (\pi_{u,0})_{\ast}\Omega_{\tilde{U}} \stackrel{\delta_{u,1}}{\la}  (\pi_{u,1})_{\ast}\Omega_{\tilde{U}_{[1]}}
\stackrel{\delta_{u,2}}{\la} \cdots
\end{equation}
where the coboundary maps are the \v{C}ech maps corresponding to the numbering of the components of $U$ induced from the numbering of the components of $W$. This concludes the proof of the claim.

Next we claim that \'etale neighborhoods of $X$ form a basis for the \'etale topology of $X$. This means that for any \'etale map $f \colon Y \rightarrow X$ and a point $y \in Y$, there exists an e\'tale  neighborhood $g \colon (U,u) \rightarrow (X,x)$, $x=f(y)$, and a factorization $h \colon U \rightarrow Y$ such that $fh=g$ and $h(u)=y$. Indeed, if $h \colon (U,u) \rightarrow (Y,y)$ is an \'etale neighborhood of $(Y,y)$, then $fh \colon (U,u) \rightarrow (X,x)$ is an \'etale neighborhood of $(X,x)$.

Let $E^{\bullet}_U$ denote the exact sequence~(\ref{local-seq2}) corresponding to the \'etale neighborhood $f \colon U \rightarrow X$. Then since \'etale neighborhoods form a basis for the \'etale topology of $X$, descent theory says that in order to construct an exact sequence on $X$ which pulls back to  $E^{\bullet}_U$, it suffices to construct for any $X$-map $\Phi_{vu} \colon V \rightarrow U$ between \'etale neighborhoods $f\colon V \rightarrow X$ and $g \colon U \rightarrow X$ of $X$, exact sequence isomorphisms 
\[
 \Psi_{vu} \colon \Phi_{vu}^{\ast}(E^{\bullet}_U) \rightarrow E^{\bullet}_V
\]
such that for any commutative diagram
\[
\xymatrix{
  & U \ar[dl]_{\Phi_{uv}}\ar[dr]^{\Phi_{uw}} & \\
V\ar[rr]^{\Phi_{vw}} & & W
}
\]
of \'etale neighborhoods of $X$, the following diagram commutes
\[
\xymatrix{
  &  E^{\bullet}_U \ar[dl]_{\Psi_{uv}}\ar[dr]^{\Psi_{uw}} & \\
\Phi^{\ast}_{uv}E^{\bullet}_V \ar[rr]^{\Phi_{uw}^{\ast}(\Psi_{vw})} & & \Phi_{uv}^{\ast}\Phi^{\ast}_{vw}E^{\bullet}_W
}
\]
In order to have a uniform numbering of the irreducible components of all \'etale neighborhoods, given an \'etale neighborhood $f \colon W \rightarrow X$ with irreducible components $W_1, \ldots, W_{r(w)}$, we extend the definition of $W_i$ for all $i \in \{1,2,\ldots , n\}$, by setting $W_i =\emptyset$, for $n \geq i > r(w)$, where $n=\dim X+1$. 

Let $\Phi \colon U \rightarrow V$ be a map between two \'etale neighborhoods of $X$, $U \stackrel{f}{\rightarrow} X$ and $V \stackrel{g}{\rightarrow} X$. 
Let $U=\cup_{i=1}^nU_i$, $V=\cup_{j=1}^nV_j$ be the decompositions of $U$ and $V$ into irreducible components, taking into account the conventions on the irreducible components stated in the previous paragraph.. Then \[
U=\cup_{i=1}^nU_i= \cup_{j=1}^n\Phi^{-1}(V_j).
\]
Moreover, since all irreducible components of $U$ pass through the same point, $\Phi^{-1}(V_j)$ is irreducible, or otherwise it would be a disjoint union of smooth irreducible components of $U$, which is impossible since all irreducible components of $U$ intersect. Therefore there exists a permutation $\sigma \in S_{n+1}$ such that 
$U_{\sigma(i)}=\Phi^{-1}(V_{i})$, $1 \leq i \leq n$.

Let $\pi_k \colon \tilde{U}_{[k]} \rightarrow U_{[k]}$ and  $\nu_k \colon \tilde{V}_{[k]} \rightarrow V_{[k]}$ be as in Definition~\ref{sing-k}. Then
\begin{eqnarray*}
(\pi_k)_{\ast}\mathcal{O}_{\tilde{U}_{[k]}}=\bigoplus_{i_0<i_1<\cdots <i_k}\mathcal{O}_{U_{i_0i_1\cdots i_k}} & \text{and} & 
(\nu_k){\ast}\mathcal{O}_{\tilde{V}_{[k]}}=\bigoplus_{i_0<i_1<\cdots <i_k}\mathcal{O}_{V_{i_0i_1\cdots i_k}}
\end{eqnarray*}
where $U_{i_0i_1\cdots i_k}=U_{i_0}\cap \cdots \cap U_{i_k}$, and similarly for $V_{i_0\cdots i_k}$. Let $U^{\prime}_i=\Phi^{-1}(V_i)$ and define the map
\begin{equation}\label{lambda-k}
\lambda_k \colon \bigoplus_{i_0<i_1<\cdots < i_k}\mathcal{O}_{U_{i_0i_1\cdots i_k}} \rightarrow \bigoplus_{i_0<i_1<\cdots < i_k}\mathcal{O}_{U^{\prime}_{i_0i_1\cdots i_k}}
\end{equation}
by setting for any $\alpha \in \oplus_{i_0<i_1<\cdots <i_k}\mathcal{O}_{U_{i_0i_1\cdots i_k}}$
\[
\lambda_k(\alpha)_{i_0\cdots i_k}=\mathrm{sgn}(\tau)\alpha_{\tau\sigma(i_0)\cdots\tau\sigma(i_k)}
\]
where $\tau \in S_{k+1}$ is the permutation such that $\tau\sigma(i_0) < \tau\sigma(i_1) < \cdots < \tau\sigma(i_k)$. 

Note that both sides of~\ref{lambda-k} are isomorphic to $(\pi_k)_{\ast}\mathcal{O}_{\tilde{U}_{[k]}}$. On the left hand side it is written by using the ordering of the components of $U$ coming from its structure as an \'etale neighborhood of $X$ while the right hand side is using the ordering of the components of $U$ inherited from $V$ by $\Phi$. By this consideration, $\lambda_k$ gives an automorphism of $U_{[k]}$ such that $\lambda_k^2$ is the identity. Moreover, it is straightforward to check that the diagram
\[
\xymatrix{
\bigoplus_{i_0<\cdots <i_k}\mathcal{O}_{U_{i_0i_1\cdots i_k}} \ar[d]_{\lambda_k}\ar[r]^{d_k} & \bigoplus_{i_0<\cdots <i_k}\mathcal{O}_{U_{i_0i_1\cdots i_{k+1}}} \ar[d]^{\lambda_{k+1}}\\
\bigoplus_{i_0<\cdots <i_k}\mathcal{O}_{U^{\prime}_{i_0i_1\cdots i_k}} \ar[r]^{\delta_k} & \bigoplus_{i_0<\cdots <i_k}\mathcal{O}_{U^{\prime}_{i_0i_1\cdots i_{k+1}}} 
}
\]
commutes. Therefore $\lambda_k$ gives a map between \v{C}ech complexes. Let \[
(\nu_k)_{\ast}\Omega_{V_{[k]}}\stackrel{\delta_{V,k}}{\rightarrow} (\nu_{k+1})_{\ast}\Omega_{V_{[k+1]}}
\]
be the map at the $k$ stage of the exact sequence $E^{\bullet}_{V}$. This pulls back by $\Phi$ to a map 
\[
(\pi_k)_{\ast}\Omega_{U_{[k]}}\stackrel{\delta_{V,k}}{\rightarrow} (\nu_{k+1})_{\ast}\Omega_{U_{[k+1]}}.
\]
This is simply the map of \v{C}ech complexes corresponding to the ordering of the irreducible components of $U$ induced from the ordering of $V$ by $\phi$. Moreover, $\lambda_k$ induces an isomorphim
\[
\Lambda_k \colon (\pi_k)_{\ast}\Omega_{U_{[k]}} \rightarrow (\pi_k)_{\ast}\Omega_{U_{[k]}}
\]
such that $\Lambda_k^2$ is the identity. Moreover a straightforward calculation shows that there is a commutative diagram
\begin{equation}\label{glueing-diag}
\xymatrix{
(\pi_k)_{\ast}\Omega_{U_{[k]}}\ar[d]_{\Lambda_k}\ar[r]^{\delta_{U,k}} & (\pi_{k+1})_{\ast}\Omega_{U_{[k+1]}} \ar[d]^{\Lambda_{k+1}}\\
\pi_k)_{\ast}\Omega_{U_{[k]}}\ar[r]^{\delta_{V,k}} & (\pi_{k+1})_{\ast}\Omega_{U_{[k+1]}} 
}
\end{equation}
Now by descent theory, the involutions $\lambda_k$ glue the structure sheaves $\mathcal{O}_{\tilde{V}_{[k]}}$ to 2-torsion sheaves $L_k$ on $\tilde{X}_{[k]}$. Then from the diagram~(\ref{glueing-diag}) all the maps in the diagram glue as well and we get a sequence on $X$
\[
0 \la \tau_X \la \Omega_X \la (\pi_0)_{\ast}\Omega_{\tilde{X}} \stackrel{\delta_1}{\la}  (\pi_1)_{\ast}(\Omega_{\tilde{X}_{[1]}}\otimes L_1)
\stackrel{\delta_2}{\la} (\pi_2)_{\ast}(\Omega_{\tilde{X}_{[2]}}\otimes L_2) \la \cdots
\]
This sequence is exact since it pulls back locally by faithful flat maps (that correspond to local \'tale neighborhoods) to exact sequences. This concludes the proof of the first part of Theorem~\ref{ob3}.

Next we sketch the proof of the other parts of the theorem. First we will construct the exact sequences~(\ref{ob3}.2), (\ref{ob3}.3) in the locally and then glue them to get the global sequence, exactly as in the proof of~(\ref{ob3}.1). Locally in the e\'tale topology 
\[
X=\mathrm{Spec}\frac{k[x_0,\ldots,x_n]}{(x_0\cdots x_r)}
\]
In this case $X=\cup_{i=1}^r X_r$, where $X_i$ is given by $x_i=0$. Then, 
\[
(\pi_i)_{\ast}\mathcal{O}_{\tilde{X}_i}=\bigoplus_{j_0<\cdots <j_i}\mathcal{O}_{X_{j_0} \cap \cdots \cap X_{j_i}}
\]
Then we define the sequence of maps 
\begin{gather}\label{local-log-complex}
0 \rightarrow \tau_X \rightarrow \Omega_X  \stackrel{\lambda_0}{\rightarrow} \Omega_{X}(log) \stackrel{\lambda_1}{\rightarrow} (\pi_1)_{\ast} \mathcal{O}_{\tilde{X}_{[1]}}  \stackrel{\lambda_2}{\rightarrow} \cdots \stackrel{\lambda_{i}}{\rightarrow} (\pi_i)_{\ast}\mathcal{O}_{\tilde{X}_{[i]}}  \stackrel{\lambda_{i+1}}{\rightarrow} (\pi_{i+1})_{\ast}\mathcal{O}_{\tilde{X}_{[i+1]}} \rightarrow \cdots \\
0 \rightarrow \mathcal{O}_X \rightarrow (\pi_0)_{\ast}\mathcal{O}_{\tilde{X}} \rightarrow (\pi_1)_{\ast} \mathcal{O}_{\tilde{X}_{[1]}}  \rightarrow \cdots \rightarrow (\pi_i)_{\ast}\mathcal{O}_{\tilde{X}_{[i]}}  \rightarrow (\pi_{i+1})_{\ast}\mathcal{O}_{\tilde{X}_{[i+1]}} \rightarrow \cdots 
 \end{gather}
as follows. The second sequence is simply the \v{C}ech coboundary maps. The maps $\lambda_i$ are the \v{C}ech coboundary maps for $i \geq 2$ and $\lambda_0$ is the natural map between $\Omega_X$ and $\Omega_X(log)$. $\Omega_X(log)$ is a free $\mathcal{O}_X$-module generated by $dx_1/x_1,\ldots, dx_r/x_r,dx_{r+1},\ldots,dx_n$ with the relation $dx_0/x_0+\cdots dx_r/x_r=0$. Then we define $\lambda_1(dx_i)=0$, if $i>r$ and if $i\leq r$, $\lambda_1(dx_i/x_i)=(\alpha_{j_0,j_1}), j_0<j_1$ such that
\[
\alpha_{j_0,j_1}= 
\begin{cases}
1 & \text{if $j_0=i$ } \\
-1 & \text{ if $j_1=i$}\\
0 & \text{otherwise}
\end{cases}
\] 
Now by either using the same method as in the first part for the sequence~(\ref{ob-seq}) or by~\cite[Corollary 3.6]{Fr83}, we get that~(\ref{local-log-complex}) and (4.6.8) are exact. Now by using exactly the same argument as in the first part by using \'etale covers we get the existence of~(\ref{ob3}.2) and (\ref{ob3}.3).

\end{proof}

\begin{theorem}\label{ob4}
Let $X$ be a Fano variety with normal crossing singularities defined over a field $k$ of characteristic zero. Then \[
H^2(T_X)=H^2(X,\mathcal{O}_X)=0.
\]
Moreover, if $X$ has a semistable logarithmic structure, then \[
H^2(\mathcal{H}om(\Omega_X(log),\mathcal{O}_X))=0
\]
as well.
\end{theorem}
\begin{corollary}\label{effectivity}
Let $X$ be a Fano variety defined over an algebraically closed field of characteristic zero with normal crossing singularities. Then any formal deformation of $X$ is effective.
\end{corollary}
\begin{proof}
This follows since $H^2(X,\mathcal{O}_X)=0$ and Grothendiecks criterion of effectivity~\cite[Theorem 2.5.13]{Ser06}.
\end{proof}

Theorem ~\ref{ob4} and Theorems~\ref{hull},~\ref{log-obstructions} imply that
\begin{corollary}\label{smooth-hull}
Let $X$ be a Fano variety with normal crossing singularities defined over a field $k$ of characteristic zero. Assume that $X$ has a semistable logarithmic structure $\mathcal{M}$. Then $\mathrm{LD}(X,\mathcal{M})$ is smooth.
\end{corollary}

\begin{proof}[Proof of Theorem~\ref{ob4}]
We only prove the vanishing of $H^2(T_X)$. The remaining parts of the theorem are proved in exactly the same way by using the exact sequences~(\ref{ob3}.2) and ~(\ref{ob3}.3.) and in addition that by Serre duality, $H^2(X,\mathcal{O}_X)=H^{n-2}(X,\omega_X)$, where $n = \dim X$.

By Corollary~\ref{ob2},
\begin{equation}\label{ob-eq5}
H^2(T_X)=H^{n-2}((\Omega_X/\tau_X)\otimes \omega_X)
\end{equation}
where $n=\dim X$. Then by Theorem~\ref{ob3}, there is an exact sequence
\begin{gather}\label{ob-seq4}
0 \la \Omega_X/\tau_X \stackrel{\delta_0}{\la} (\pi_0)_{\ast}\Omega_{\tilde{X}} \stackrel{\delta_1}{\la}  (\pi_1)_{\ast}(\Omega_{\tilde{X}_{[1]}}\otimes L_1)
\stackrel{\delta_2}{\la} \cdots \stackrel{\delta_N}{\la}(\pi_N)_{\ast}(\Omega_{\tilde{X}_{[N]}}\otimes L_N) \la 0
\end{gather}
where $N \leq \dim X$, $\tau_X$ is the torsion part of $\Omega_X$ and $L_i$ is an invertible sheaf on $\tilde{X}_{[i]}$ such that $L_i^{\otimes 2}
\cong \mathcal{O}_{\tilde{X}_{[i]}}$.

Tensoring~(\ref{ob-seq4} ) with $\omega_X$ and taking into consideration that 
\[
(\pi_i)_{\ast}(\Omega_{\tilde{X}_{[i]}} \otimes (\pi_i)^{\ast}\omega_X)=(\pi_i)_{\ast} \Omega_{\tilde{X}_{[i]}} \otimes \omega_X
\]
we get the exact sequence
\begin{gather*}
0 \la (\Omega_X/\tau_X) \otimes \omega_X  \stackrel{\delta_0}{\la} (\pi_0)_{\ast}(\Omega_{\tilde{X}}\otimes \pi_0^{\ast}\omega_X) \stackrel{\delta_1}{\la}  (\pi_1)_{\ast}(\Omega_{\tilde{X}_{[1]}}\otimes L_1 \otimes \pi_1^{\ast}\omega_X) 
\stackrel{\delta_2}{\la} \cdots \\
\cdots \stackrel{\delta_N}{\la}(\pi_N)_{\ast}(\Omega_{\tilde{X}_{[N]}}\otimes L_N\otimes \pi_N^{\ast}\omega_X) \la 0
\end{gather*}

Let $M_k=\mathrm{Im}(\delta_k)$, $1 \leq k \leq N$. Then the above sequence splits into
\begin{gather*}
0 \la (\Omega_X/\tau_X)\otimes \omega_X  \stackrel{\delta_1}{\la} (\pi_0)_{\ast}(\Omega_{\tilde{X}}\otimes \pi_0^{\ast}\omega_X) \stackrel{\delta_2}{\la} M_1\la 0\\
0 \la M_k \la (\pi_k)_{\ast}(\Omega_{\tilde{X}_{[k]}}\otimes L_k\otimes \pi_k^{\ast}\omega_X) \la M_{k+1} \la 0
\end{gather*}
where $1 \leq k \leq N-1$, $N \leq n=\dim X$ and $M_N=(\pi_N)_{\ast}(\Omega_{\tilde{X}_{[N]}}\otimes L_N\otimes \pi_N^{\ast}\omega_X)$. Therefore we get exact sequences in cohomology
\begin{gather}\label{eq-ob4-1}
\cdots H^{n-3}(M_1) \la H^{n-2}( (\Omega_X/\tau_X)\otimes \omega_X ) \la H^{n-2}((\pi_0)_{\ast}(\Omega_{\tilde{X}}\otimes \pi_0^{\ast}\omega_X)) \la \cdots \\
\cdots H^{n-k-3}(M_{k+1}) \la H^{n-k-2}(M_k) \la H^{n-k-2}((\pi_k)_{\ast}(\Omega_{\tilde{X}_{[k]}}\otimes L_k \otimes \pi_k^{\ast}\omega_X)) \la \cdots\nonumber
\end{gather}

Now since $\pi_k$ are finite, it follows that $(\pi_k^{\ast}\omega_X)^{-1}$ are ample, for all $0 \leq k \leq N$ and hence $(L_k^{-1} \otimes\pi_k ^{\ast}\omega_X)^{-1}$ is ample too, since $L_k$ is 2-torsion and invertible. 
Moreover, $\tilde{X}_{[k]}$ is smooth of dimension $n-k$.  Therefore, and by using the Kodaira-Nakano vanishing theorem~\cite[Corollary 6.4]{EV92},
\[
H^{n-k-2}((\pi_k)_{\ast}(\Omega_{\tilde{X}_{[k]}}\otimes L_k \otimes \pi_k^{\ast}\omega_X)) =0
\]
for all $1 \leq k \leq N$. Hence from~(\ref{eq-ob4-1}) and by induction it follows that $H^{n-k-2}(M_k)=0$, for all $0 \leq k \leq N$ and hence again by~(\ref{eq-ob4-1}) it follows that there is an exact sequence
\[
H^{n-3}(M_1) \la H^{n-2}((\Omega_X/\tau_X)\otimes \omega_X) \la H^{n-2}((\pi_0)_{\ast}(\Omega_{\tilde{X}}\otimes \pi_0^{\ast}\omega_X)) 
\]
and therefore 
\[
H^2(T_X)=H^{n-2}((\Omega_X/\tau_X)\otimes \omega_X) =0
\]
as claimed. 

\end{proof}

Unfortunately, in general I cannot say much about the other obstruction space, namely $H^1(T^1_X)$. However, since $T^1_X$ is a line bundle on the singular locus $X_{[1]}$ of $X$, it is much easier handled than $H^2(T_X)$ and it will vanish if we impose certain positivity requirements on $T^1_X$.

The case when $X$ has only double points exhibits much better behavior and it deserves special consideration. The difference between this and the general case is that the singular locus $X_{[1]}$ of $X$ is smooth.

\begin{theorem}\label{db-ob}
Let $X$ be a Fano variety with only double point normal crossing singularities such that $T^1_X$ is finitely generated by its global sections. Then
\[
H^2(T_X)=H^1(T^1_X)=0
\]
\end{theorem}
\begin{corollary}
Let $X$ be a Fano variety with only double point normal crossing singularities and such that $T^1_X$ is finitely generated by its global sections. Then $\mathrm{Def}(X)$ is smooth.
\end{corollary}
\begin{question}
Is $Def(X)$ smooth for any Fano variety with normal crossing singularities? If this is true then $X$ is smoothable if and only if $T^1_X$ is finitely generated by global sections and hence this is a very natural condition to impose.
\end{question}
\begin{remark}
In general, $H^1(T^1_X)$ will not vanish. However if $X$ is smoothable, then $T^1_X$ must have some positivity properties and the one stated is the most natural one.
\end{remark}
\begin{proof}[Proof of Theorem~\ref{db-ob}]
In view of Theorem~\ref{ob4} we only need to show the vanishing of $H^1(T^1_X)$. In order to show this we will first show that the singular locus $Z=X_{[1]}$ of $X$ is a smooth Fano variety of dimension $\dim X-1$. The Fano part is the only part to be shown. Let $\pi \colon \tilde{X} \la X$ be the normalization and $\tilde{Z}=\pi^{-1}Z$. Then $\tilde{Z} \la Z$ is \'etale. By subadjunction we get that
\[
\pi^{\ast}\omega_X=\omega_{\tilde{X}} \otimes \mathcal{O}_{\tilde{X}}(\tilde{Z})
\]
Therefore,
\[
\omega_{\tilde{Z}}=\omega_{\tilde{X}}\otimes \mathcal{O}_{\tilde{X}}(\tilde{Z})\otimes \mathcal{O}_{\tilde{Z}}
\]
Hence $\omega_{\tilde{Z}}^{-1}$ is ample. But since $\tilde{Z} \la Z$ is \'etale, it follows that $\pi^{\ast}\omega_Z=\omega_{\tilde{Z}}$. Therefore $\omega_Z^{-1}$ is ample too and hence $Z$ is Fano as claimed.  Now
\[
H^1(T^1_X)=H^1(\omega_Z \otimes(T^1_X\otimes \omega_Z^{-1}))=0
\]
by the Kawamata-Viehweg vanishing theorem since if $T^1_X$ is finitely generated by global sections, then $T^1_X \otimes \omega_Z^{-1}$ is ample too.

\end{proof}

\section{Smoothings of Fanos}
In this section we obtain criteria for a Fano variety $X$ with normal crossing singularities to be smoothable. First we state a criterion for a variety $X$ with hypersurface singularities to be smoothable and moreover to be smoothable with a smooth total space.
\begin{proposition}\label{sm1}
Let $X$ be a reduced projective scheme with hypersurface singularities and let $D$ be its singular locus. Then
\begin{enumerate}
\item If $X$ is smoothable by a flat deformation $\mathcal{X} \la \Delta$ such that $\mathcal{X}$ is smooth, then $T^1_X=\mathcal{O}_D$.
\item Suppose that $T^1_X$ is finitely generated by its global sections and that $H^2(T_X)=H^1(T^1_X)=0$. Then $X$ is smoothable. Moreover, if $Def(X)$ is smooth then the converse is also true.
\end{enumerate}
\end{proposition}

\begin{proof}
The second part is~\cite[Theorem 12.5]{Tz10}. We proceed to show the first part. Let $f \colon \mathcal{X} \la \Delta$ be a smoothing of $X$ such that $\mathcal{X}$ is smooth, where $\Delta=\mathrm{Spec}(R)$, $(R,m_R)$ is a discrete valuation ring. Let $T^1_{\mathcal{X}/\Delta}=\mathcal{E}xt^1_{\mathcal{X}}(\Omega_{\mathcal{X}/\Delta},\mathcal{O}_{\mathcal{X}})$ be Schlessinger's relative $T^1$ sheaf. Then dualizing  the exact sequence
\[
0 \la f^{\ast}\omega_{\Delta} = \mathcal{O}_{\mathcal{X}} \la \Omega_{\mathcal{X}} \la \Omega_{\mathcal{X}/\Delta} \la 0
\]
we get the exact sequence
\[
\cdots \la \mathcal{O}_{\mathcal{X}} \la T^1_{\mathcal{X}/\Delta} \la \mathcal{E}xt^1_{\mathcal{X}}(\Omega_{\mathcal{X}},\mathcal{O}_{\mathcal{X}}) =0
\]
Now restricting to the special fiber and taking into consideration that $\mathcal{X} \otimes_R R/m_r \cong X$ and that $T^1_{\mathcal{X}/\Delta}\otimes_R R/m_R=T^1_X$~\cite[Lemma 7.7]{Tz10}, we get that there is a surjection $\mathcal{O}_X \la T^1_X$. Moreover, $T^1_X$ is a line bundle on the singular locus $D$ of $X$. Hence, restricting on $Z$ it follows that $T^1_X \cong \mathcal{O}_D$, as claimed.
\end{proof}

\begin{remark}
The condition $T^1_X=\mathcal{O}_D$ is equivalent to Friedman's d-semistability condition in the case of reducible simple normal crossing schemes~\cite{Fr83}. One of the natural questions raised by Friedman is whether this condition is sufficient for a simple normal crossing variety to be smoothable. He showed that in the case of $K3$ surfaces this is true but Persson and Pinkham have shown that this is not true in general~\cite{PiPe83}. However this is true in the case of normal crossing (not necessarily reducible) Fano schemes, as shown by Theorem~\ref{sm2} below.
\end{remark}

\begin{theorem}\label{sm2}
Let $X$ be a Fano variety defined over an algebraically closed field of characteristic zero with normal crossing singularities. Assume that one of the following conditions holds:
\begin{enumerate}
\item $T^1_X$ is finitely generated by global sections and  $H^1(T^1_X)=0$.
\item $X$ has at worst double point normal crossing singularities and that $T^1_X$ is finitely generated by global sections.
\item $X$ is $d$-semistable, i.e., $T^1_X\cong \mathcal{O}_D$, where $D$ is the singular locus of $X$.
\end{enumerate}
Then $X$ is smoothable. Moreover, $X$ is smoothable by a flat deformation $f \colon \mathcal{X} \rightarrow \Delta$ such that $\mathcal{X}$ is smooth, if and only if $X$ is $d$-semistable.
\end{theorem}

\begin{proof}
(\ref{sm2}.1) follows directly from Theorem~\ref{ob4} and Proposition~\ref{sm1}.2.~(\ref{sm2}.2) follows from Theorems~\ref{db-ob} and~\ref{sm1}.2. Finally, suppose that $T^1_X\cong \mathcal{O}_D$, i.e., $X$ is $d$-semistable. Then according to Proposition~\ref{d-semistable}, $X$ admits a semi-stable logarithmic structure $\mathcal{M}$. Moreover, by Theorem~\ref{ob4} and Corollary~\ref{smooth-hull}, $(X,\mathcal{M})$ has unobstructed logarithmic deformations. Let $s$ be a nowhere vanishing section of $T^1_X$, as a sheaf on $D$. Then locally in the \'etale topology, \[
X=\mathrm{Spec}\frac{k[x_0,\ldots, x_d]}{(x_0\cdots x_r)}
\]
and $s$ corresponds to the first order deformation of $X$
\[
X_1=\mathrm{Spec}\frac{A_1[x_0,\ldots, x_d]}{(x_0\cdots x_r-t)} \rightarrow \mathrm{Spec} A_1
\]
where $A_1=k[t]/(t^2)$. This deformation is also a log-deformation for the semistable log structure of $X$. This is evident from the diagram
\[
\xymatrix{
\mathbb{N}^d \ar[r]^{\alpha}  & \frac{A_1[x_0,\ldots, x_d]}{(x_0\cdots x_r-t)} \\
\mathbb{N} \ar[u]^{\Delta} \ar[r]^{\beta} & A_1 \ar[u] 
}
\]
where $\Delta$ is the diagonal, $\alpha$ the semistable log-structure of $X$ and $\beta(n)=t^n$. Since $X$ has unobstructed log-deformations, this first order deformation lifts to a formal log-deformation of $X$ over $k[[t]]$. By Corollary~\ref{effectivity} any formal deformations is effective. Let then $f \colon (\mathcal{X},\mathcal{N}) \la (\Delta, \mathbb{N})$, where $\Delta=\mathrm{Spec}k[[t]]$ be the lifting. Then as in the proof of Proposition~\ref{sm1}, there is an exact sequence
\[
\cdots \la \mathcal{O}_{\mathcal{X}} \stackrel{\sigma}{\la} T^1_{\mathcal{X}/\Delta} \la \mathcal{E}xt^1_{\mathcal{X}}(\Omega_{\mathcal{X}},\mathcal{O}_{\mathcal{X}}) \rightarrow 0
\]
Now by the assumption that $T^1_X=\mathcal{O}_D$, and the non-triviality of the deformation, it follows that $\sigma$ is surjective and hence 
\[
\mathcal{E}xt^1_{\mathcal{X}}(\Omega_{\mathcal{X}},\mathcal{O}_{\mathcal{X}})=0
\]
Therefore, since $\mathcal{X}$ has complete intersection singularities, $\mathcal{X}$ is smooth.

\end{proof}

\section{Examples.}
In this section we construct one example of a smoothable and one of a non-smoothable normal crossing Fano 3-fold.

\begin{example}\label{ex1} Let $P \in Y \subset \mathbb{P}^4$ be a quadric surface with one ordinary double point locally analytically isomorphic to $(xy-zw=0)\subset \mathbb{C}^4$. Let $f \colon X \la Y$ be the blow up of $P \in Y$. Then $X$ is smooth and the $f$-exceptional divisor $E$ is isomorphic to $\mathbb{P}^1 \times \mathbb{P}^1$. Moreover, $-K_X-E$ is ample and $\mathcal{N}_{E/X}=\mathcal{O}_E(-1,-1)$.

Next we construct an embedding $E \subset X^{\prime}$ of $E$ into a smooth scheme $X^{\prime}$ such that, $\mathcal{N}_{E/X^{\prime}}=\mathcal{O}_{E}(1,1)$ and $-K_{X^{\prime}}-E$ is ample. Let $Z \subset \mathbb{P}^3$ be a smooth quadric surface. Then $\mathcal{N}_{Z/\mathbb{P}^3}=\mathcal{O}_Z(2,2)$. Let $\pi \colon X^{\prime} \la \mathbb{P}^3$ be the cyclic double cover of $\mathbb{P}^3$ ramified over $Z$. This is defined by the line bundle $\mathcal{L}=\mathcal{O}_{\mathbb{P}^3}(1)$ and the section $s$ of $\mathcal{L}^{\otimes 2}$ that corresponds to $Z$. Let $E=(\pi^{-1}(Z))_{\mathrm{red}} \cong Z$. Then $\pi^{\ast}Z = 2E$ and $\omega_{X}=\pi^{\ast}(\omega_{\mathbb{P}^3}\otimes \mathcal{L})$. Let $l^{\prime} \subset E$ be one of the rulings and $l = \pi_{\ast}(l^{\prime})$. then
\[
l^{\prime} \cdot E = 1/2 (l^{\prime} \cdot \pi^{\ast} Z) = 1/2( l \cdot Z) = 1
\]
and hence $\mathcal{N}_{E/\mathcal{X}^{\prime}}=\mathcal{O}_{E}(1,1)$.  Now let $Y$ be the scheme obtained by glueing $X$ and $X^{\prime}$ along $E$. This is a normal crossing Fano 3-fold with only double points. Then $T^1_Y= \mathcal{N}_{E/X} \otimes \mathcal{N}_{E/X^{\prime}} = \mathcal{O}_{E}$. Therefore, by Theorem~\ref{sm2}, $Y$ is smoothable.
\end{example}

\begin{example}\label{ex2} Let $E \subset X$ be as in example 1. Then let $Y$ be obtained by glueing two copies of $X$ along $E$. Then,
\[
T^1_Y=\mathcal{N}_{E/X} \otimes \mathcal{N}_{E/X} = \mathcal{O}_E(-2,-2)
\]
and hence $H^0(T^1(Y))=0$. Hence $Y$ is not smoothable~\cite[Theorem 12.3]{Tz09}. In fact, every deformation of $Y$ is locally trivial.
\end{example}


\begin{thebibliography}{Lichtenbaum}
\bibitem[AN54]{AN54} Y. Akizuki, S. Nakano \textit{Note on Kodaira-Spencer's proof of Lefschetz's theorem}, Proc. Jap. Acad., Ser A 30, 1954, 266-272.
\bibitem[Art69]{Art69} M. Artin, \textit{Algebraic approximation of structures over complete local rings}, Publ. Math. IHES 36, 1969, 23-58.
\bibitem[EV92]{EV92} E. Esnault, E. Viehweg, \textit{Lectures on Vanishing Theorems}, Birkh\"auser Verlag, 1992.
\bibitem[BLLSNA]{BLLSNA} B. Fantechi, L. G\"ottsche, L. Illusie, S. Kleiman, N. Nitsure, A. Vistoli, \textit{Fundamental Algebraic Geometry: Grothendieck's FGA explained}, Mathematical Surveys and Monographs, vol. 123, American Mathematical Society, 2005.
\bibitem[Fr83]{Fr83} R. Friedman, \textit{Global smoothings of varieties with normal crossings}, Ann. of Math. 118, 1983, 75-114.
\bibitem[Fu90]{Fu90} T. Fujita, \textit{On Del Pezzo fibrations over curves}, Osaka J. Math. 27, 1990, 229-245.
\bibitem[Kac07]{Kac07} Y. Kachi, \textit{Global smoothings of degenerate del Pezzo surfaces with normal crossings}, J. of Algebra 307, 2007, 249-253.
\bibitem[KawNa94]{KawNa94} Y. Kawamata, Y. Namikawa, \textit{Logarithmic deformations of normal crossing varieties and smoothing of degenerate Calabi-Yau varieties}, Invent. Math. 118, 1994, 395-409.
\bibitem[KK88]{Ka88} K. Kato, \textit{Logarithmic structures of Fontaine-Illusie}, in Algebraic Analysis Geometry and Number Theory, 1988, Johns Hopkins University, 191-224.
\bibitem[FK96]{Kat96} F. Kato, \textit{Log smooth deformation theory}, Tohoku Math. J. 48, 1996, 317-354.
\bibitem[KKMS73]{KKMS73} G.Kempf, F. Knudsen, D. Mumford, B. Saint Donat, \textit{Toroidal Embeddings I}, Lect. Notes Math. Vol. 339, Springer-Verlag, Berlin Heidelberg, 1973.
\bibitem[Li-Sch67]{Li-Sch67} S. Lichtenbaum, M. Schlessiger, \textit{ The cotangent complex of a morphism}, Trans. Amer. Math. Soc. 128, 1967, 41-70.
\bibitem[Mil80]{Mil80} J. Milne, \textit{\'Etale Cohomology}, Princeton University Press, 1980.
\bibitem[PiPe83]{PiPe83} H. Pinkham, U. Persson, \textit{Some examples of nonsmoothable varieties with normal crossings}, Duke Math. J. 50, 1983, 477-486.
\bibitem[Ser06]{Ser06} E. Sernesi, \textit{Deformations of Algebraic Schemes}, Springer-Verlag Berlin Heidelberg, 2006.
\bibitem[Sch68]{Sch68} M. Schlessinger, \textit{Functors of Artin rings}, Trans. AMS 130, 1968, 208-222.
\bibitem[Tz09]{Tz09} N. Tziolas, \textit{$\mathbb{Q}$-Gorenstein smoothings of nonnormal surfaces}, Amer. J. Math.  131  (2009),  no. 1, 171-193.
\bibitem[Tz10]{Tz10} N. Tziolas, \textit{Smoothings of schemes with nonisolated singularities}, Michigan Math. J. Volume 59, Issue 1 (2010), 25-84.
\end{thebibliography}
\end{document}